\documentclass[a4paper]{amsart}
\usepackage[leqno]{amsmath}
\usepackage{amssymb}
\usepackage{amscd}
\usepackage{amsthm}
\usepackage{mathrsfs}
\usepackage{mathtools}
\usepackage{bbm}
\usepackage{color}
\usepackage{enumerate}
\usepackage{cite}
\usepackage[utf8]{inputenc}
\usepackage[all,cmtip]{xy}
\usepackage{etoolbox}
\usepackage{tikz-cd}
\usepackage{extarrows}

\numberwithin{equation}{section}

\newcommand{\Z}{\ensuremath{\mathbb{Z}}}
\newcommand{\Q}{\ensuremath{\mathbb{Q}}}

\newcommand{\C}{\ensuremath{\mathbb{C}}}

\newcommand{\A}{\ensuremath{\mathbb{A}}}
\newcommand{\T}{\ensuremath{\mathbb{T}}}

\DeclareMathOperator{\Hom}{Hom}
\DeclareMathOperator{\End}{End}
\DeclareMathOperator{\Ext}{Ext}

\DeclareMathOperator{\id}{id}

\DeclareMathOperator{\PGL}{\mathbb{PGL}}

\DeclareMathOperator{\ord}{ord}
\DeclareMathOperator{\Ind}{Ind}

\DeclareMathOperator{\St}{St}
\DeclareMathOperator{\cind}{c-ind}

\DeclareMathOperator{\HH}{H}
\DeclareMathOperator{\MS}{\mathcal{M}}
\renewcommand{\det}{\operatorname{det}}

\newcommand{\PP}{\ensuremath{\mathbb{P}^1}}					
\newcommand{\sinfty}{\ensuremath{^{S,\infty}}}

\newcommand{\Ah}{\ensuremath{\mathcal{A}}}

\DeclareMathOperator{\Div}{Div}
\DeclareMathOperator{\cont}{cont}
\newcommand{\n}{\ensuremath{\mathfrak{n}}}
\newcommand{\m}{\ensuremath{\mathfrak{m}}}
\newcommand{\p}{\ensuremath{\mathfrak{p}}}
\newcommand{\q}{\ensuremath{\mathfrak{q}}}

\newcommand{\LI}{\mathcal{L}}
\newcommand{\No}{\ensuremath{N}}

\DeclareMathOperator{\an}{an}
\DeclareMathOperator{\pr}{pr}
\newcommand{\Hec}{\mathscr{H}}
\DeclareMathOperator{\ev}{ev}
\DeclareMathOperator{\cores}{cores}
\DeclareMathOperator{\res}{res}

\newcommand{\into}{\hookrightarrow}

\newcommand{\too}{\longrightarrow}								
\newcommand{\mapstoo}{\longmapsto}

\newtheorem{Lem}{Lemma}[section]
\makeatletter
\newlength{\@thlabel@width}%
\newcommand{\thmenumhspace}{\settowidth{\@thlabel@width}{\itshape1.}\sbox{\@labels}{\unhbox\@labels\hspace{\dimexpr-\leftmargin+\labelsep+\@thlabel@width-\itemindent}}}
\makeatother
\newtheorem{Pro}[Lem]{Proposition}
\newtheorem{MLem}[Lem]{Main Lemma}
\newtheorem{Thm}[Lem]{Theorem}

\newtheorem{Def}[Lem]{Definition}
\newtheorem{Rem}[Lem]{Remark}

\newtheorem*{theConj}{Conjecture}
\newtheorem{Cor}[Lem]{Corollary}

\author[L. Gehrmann]{Lennart Gehrmann}
\address{L. Gehrmann \\ Fakult\"at f\"ur Mathematik \\ Universit\"at Duisburg-Essen \\ Thea-Leymann-Stra\ss e 9 \\ 45127 Essen \\ Germany}
\email{lennart.gehrmann@uni-due.de}

\title[Derived Hecke algebra and automorphic $\LI$-invariants]{Derived Hecke algebra and automorphic L-invariants}

\begin{document}

\begin{abstract}
Let $\pi$ be a cohomological cuspidal automorphic representation of $PGL_2$ over a number field of arbitrary signature.
Under the assumption that the local component of $\pi$ at a prime $\p$ is the Steinberg representation the automorphic $\LI$-invariant of $\pi$ at $\p$ has been defined using the lowest degree cohomology in which the system of Hecke eigenvalues associated to $\pi$ occurs.

In this article we define automorphic $\LI$-invariants for each cohomological degree and show that they behave well with respect to the action of Venkatesh's derived Hecke algebra.
As a corollary we show that these $\LI$-invariants are (essentially) the same if the following conjecture of Venkatesh holds: the $\pi$-isotypic component of the cohomology is generated by the minimal degree cohomology as a module over the $p$-adic derived Hecke algebra. 
\end{abstract}

\maketitle

\tableofcontents

\section{Introduction}
The system of Hecke eigenvalues associated to a Hilbert modular newform of parallel weight $(2,\ldots,2)$ only shows up in the middle degree cohomology of the corresponding Hilbert modular variety.
On the contrary, if one considers an automorphic form $f$ of parallel weight $(2,\ldots,2)$ of the group $PGL_2$ over a number field $F$ which is not totally real, its system of Hecke eigenvalues show up in several degrees.
For example, assume that $F$ is an imaginary quadratic field with class number one and that the field of definition of $f$ is $\Q$.
Let $\mathfrak{n}\subseteq \mathcal{O}_F$ be the level of $f$ and let $\Gamma_0(\mathfrak{n})\subseteq PSL_2(\mathcal{O}_F)$ be the subgroup of all matrices which are congruent to an upper triangular matrix modulo $\mathfrak{n}$.
Then, we have
$$\dim \HH^{1}(\Gamma_0(\mathfrak{n}),\Q)[f]= \HH^{2}(\Gamma_0(\mathfrak{n}),\Q)[f] = 1,$$
where $[f]$ denotes the $f$-isotypic component.

In \cite{Ve} Venkatesh constructs a graded extension $\tilde{\mathbb{T}}=\oplus_{i\geq 0}\tilde{\mathbb{T}}_i$ - called the derived Hecke algebra - of the usual Hecke algebra, which acts in a graded fashion on cohomology with $p$-adic coefficients.
The classical Hecke algebra lies in the centre of $\tilde{\mathbb{T}}$ and therefore, the derived Hecke algebra maps the $f$-isotypic component of the cohomology to itself.
This action should be rich enough to account for the occurrence of the same Hecke eigenvalues in several degrees, i.e.~ the $f$-isotypic part of the cohomology should be generated by its lowest degree as a $\tilde{\mathbb{T}}$-module.
In our example, this simply means that there should exist a derived Hecke operator $t\in \tilde{\mathbb{T}}_1$ such that the corresponding map
\begin{align}\label{dream}\HH^{1}(\Gamma_0(\mathfrak{n}),\Q_p)[f]\xlongrightarrow{t} \HH^{2}(\Gamma_0(\mathfrak{n}),\Q_p)[f]\end{align}
is non-zero.
As a consequence, the different cohomology groups attached to $f$ should carry the same information, e.g.~the $p$-adic periods one can define using these different cohomology groups should be the same.

The kind of period we study in this article are the so-called automorphic $\LI$-invariants. Suppose that the automorphic form $f$ is Steinberg at a prime $\p$ over $p$, i.e.~$\p$ divides the level $\mathfrak{n}$ exactly once and the $U_\p$-eigenvalue of $f$ is $1$.
Then, using the lowest degree cohomology one can attach an automorphic $\LI$-invariant to $f$ and $\p$, which shows up in the exceptional zero formula of the $p$-adic $L$-function attached to $f$ (see \cite{Ge2} for details.)
The study of automorphic $\LI$-invariants was initiated by Darmon in \cite{D}, where he considered elliptic modular forms of weight $2$.
The construction was generalized to various settings, for example, to elliptic modular forms of higher weight by Orton (cf.~\cite{Orton}), to Hilbert modular forms of parallel weight $(2,\ldots,2)$ by Spie\ss~(cf.~\cite{Sp}), and by Barrera and Williams to Bianchi modular forms of arbitrary weight (cf.~\cite{BWi}).

Let us recall the construction of the automorphic $\LI$-invariant in our example:
let $\St_\p$ be the space of all integer-valued locally constant functions on $\PP(F_\p)$ modulo constant functions.
This is canonically a $PGL_2(F_\p)$-module.
Let $\Gamma_0^{(\p)}(\mathfrak{n})\subseteq PGL_2(\mathcal{O}_F[\p^{-1}])$ be the $\p$-arithmetic group given by the same congruence conditions as its arithmetic counterpart.
Then, the fact that $f$ is Steinberg at $\p$ implies that evaluation at an Iwahori-fixed vector induces an isomorphism
$$\HH^{1}(\Gamma_0^{(\p)}(\mathfrak{n}),\Hom(\St_\p,\Q_p))[f]\xlongrightarrow{\ev}\HH^{1}(\Gamma_0(\mathfrak{n}),\Q_p)[f]$$
of one-dimensional vector spaces.
$P$-adic integration together with Breuil's construction of extensions of the Steinberg representation yields for every continuous homomorphism $\ell\colon F_\p^{\ast}\to \Q_p$ a map
$$\HH^{1}(\Gamma_0^{(\p)}(\mathfrak{n}),\Hom(\St_\p,\Q_p))[f]\xlongrightarrow{c_{\ell}}\HH^{2}(\Gamma_0^{(\p)}(\mathfrak{n}),\Q_p)[f],$$
which turns out to be an isomorphism if $\ell=\ord_\p$ is the $p$-adic valuation.
Then, the automorphic $\LI$-invariant $\LI_{\ell}(f,\p)\in\Q_p$ is defined as the unique $p$-adic number such that
$$c_{\ell}= \LI_{\ell}(f,\p) \cdot c_{\ord_{\p}}.$$
Instead of working with the first cohomology group we could work with the second one and get an a priori different $\LI$-invariant.
But let us suppose we have an action of the derived Hecke algebra $\tilde{\mathbb{T}}$ on $\HH^{\ast}(\Gamma_0(\mathfrak{n}),\Hom(\St_\p,\Q_p))[f]$ and $\HH^{\ast}(\Gamma_0(\mathfrak{n}),\Q_p)[f]$ such that for all $t\in \tilde{\mathbb{T}}_1$ and all homomorphism $\ell\colon F_\p^{\ast}\to \Q_\p$ the diagrams
\begin{center}
\begin{tikzpicture}
    \path 	(0,0) 	node[name=A]{$\HH^{1}(\Gamma_0^{(\p)}(\mathfrak{n}),\Hom(\St_\p,\Q_p))[f]$}
		(6,0) 	node[name=B]{$\HH^{2}(\Gamma_0^{(\p)}(\mathfrak{n}),\Hom(\St_\p,\Q_p))[f]$}
		(0,-1.5) 	node[name=C]{$\HH^{1}(\Gamma_0(\mathfrak{n}),\Q_p)[f]$}
		(6,-1.5) 	node[name=D]{$\HH^{2}(\Gamma_0(\mathfrak{n}),\Q_p)[f]$};
    \draw[->] (A) -- (C) node[midway, left]{$\ev$};
    \draw[->] (A) -- (B) node[midway, above]{$t$};
    \draw[->] (C) -- (D) node[midway, above]{$t$};
    \draw[->] (B) -- (D) node[midway, right]{$\ev$};
\end{tikzpicture} 
\end{center}
and
\begin{center}
\begin{tikzpicture}
    \path 	(0,0) 	node[name=A]{$\HH^{1}(\Gamma_0^{(\p)}(\mathfrak{n}),\Hom(\St_\p,\Q_p))[f]$}
		(6,0) 	node[name=B]{$\HH^{2}(\Gamma_0^{(\p)}(\mathfrak{n}),\Hom(\St_\p,\Q_p))[f]$}
		(0,-1.5) 	node[name=C]{$\HH^{2}(\Gamma_0^{(\p)}(\mathfrak{n}),\Q_p)[f]$}
		(6,-1.5) 	node[name=D]{$\HH^{3}(\Gamma_0^{(\p)}(\mathfrak{n}),\Q_p)[f]$};
    \draw[->] (A) -- (C) node[midway, left]{$c_{\ell}$};
    \draw[->] (A) -- (B) node[midway, above]{$t$};
    \draw[->] (C) -- (D) node[midway, above]{$t$};
    \draw[->] (B) -- (D) node[midway, right]{$c_{\ell}$};
\end{tikzpicture} 
\end{center}
commute (up to a sign).
Then, the existence of an element $t\in\tilde{\mathbb{T}}_1$ as in \eqref{dream} would easily imply that the two $\LI$-invariants are equal.
The main aim of this note is to construct such actions of the derived Hecke algebra and prove that the corresponding diagrams for arbitrary number fields commute.

The paper is structured as follows:
the first part is devoted to the construction of automorphic $\LI$-invariants in arbitrary cohomological degree.
Thankfully, the construction of Spie\ss~can be easily generalized to this setup.
In Section \ref{Steinberg} we recall the construction of extensions of the Steinberg representation due to Breuil.
After that we define and give basic properties of the spaces of modular symbols, which are needed to define automorphic $\LI$-invariants.
A new feature over general number fields, which does not occur when dealing with Hilbert or Bianchi modular forms, is that those spaces of modular symbols are not necessarily one-dimensional anymore.
Thus, one is forced to define automorphic $\LI$-invariants as determinants (see Definition \ref{defi}).
We show that it does not matter whether we use cohomology with or without compact support in defining automorphic $\LI$-invariants (see Proposition \ref{compact}).
Finally, we state our conjecture that automorphic $\LI$-invariants coming from different cohomology degrees are essentially the same.

In the second part we review Venkatesh's derived Hecke algebra $\tilde{\mathbb{T}}$.
We give a wish list of properties we need for the derived Hecke algebra in order to prove our conjecture in Section \ref{algebra}.
Subsequently, we show that $\tilde{\mathbb{T}}$ fulfills most of this properties, e.g.~we show that it acts on all cohomology groups in question and that generalizations of the diagrams above commute.
Thus, we deduce that our conjecture follows if the $f$-isotypic part of the cohomology of the appropriate locally symmetric space is generated by its lowest degree as a $\tilde{\mathbb{T}}$-module (see Theorem \ref{dastheorem}).

\bigskip
\textbf{General notations.} Throughout the article we fix a prime $p$.
All rings are assumed to be commutative and unital.
The group of invertible elements of a ring $R$ will be denoted by $R^{\ast}$.

If $R$ is a ring and $G$ is a group, we write $R[G]$ for the group algebra of $G$ with coefficients in $R$.
Given a group $G$ and a group homomorphism $\epsilon\colon G\to R^{\ast}$ we let $R(\epsilon)$ be the $R[G]$-module which underlying $R$-module is $R$ itself and on which $G$ acts via the character $\epsilon$.
If $M$ is another $R[G]$-module, we put $M(\epsilon)=M\otimes_{R}R(\epsilon)$.

If $X$ and $Y$ are topological spaces, we write $C(X,Y)$ for the space of continuous maps from $X$ to $Y$.
 
\bigskip
\textbf{Acknowledgments.} The ideas presented in this article emerged during a stay at the Bernoulli Center (CIB) in the course of the semester-long program on Euler systems and Special Values of L-functions.
It is my pleasure to thank the organizers of the program as well as the local staff for a pleasant and scientifically stimulating stay.
I thank Sebastian Bartling for comments on an earlier draft of the article and Felix Bergunde for several helpful discussions.

\section{The setup}
We fix an algebraic number field $F$ of degree $d=r+2s$, where $r$ (resp.~$s$) is the number of real (resp.~complex) places of $F$.
The ring of integers of $F$ will be denoted by $\mathcal{O}$.
We denote the set of infinite places of $F$ by $S_{\infty}$.

If $v$ is a place of $F$, we denote by $F_{v}$ the completion of $F$ at $v$.
If $\mathfrak{q}$ is a finite place, we let $\mathcal{O}_{\mathfrak{q}}$ denote the valuation ring of $F_{\mathfrak{q}}$ and $\ord_{\mathfrak{q}}$ the additive valuation such that $\ord_{\mathfrak{q}}(\varpi)=1$ for any local uniformizer $\varpi\in\mathcal{O}_{\mathfrak{q}}$.
We write $\mathcal{N}(\mathfrak{q})$ for the cardinality of the residue field $\mathcal{O}/\mathfrak{q}$.

For a finite set $S$ of places of $F$ we define the "$S$-truncated adeles" $\A^{S}$ as the restricted product of all completions $F_{v}$ with $v\notin S$.
In case $S$ is the empty set we drop the superscript $S$ from the notation.
We will often write $\A\sinfty$ instead of $\A^{S\cup S_\infty}$.

We consider the algebraic group $G=PGL_2$.
If $v$ is a place of $F$, we write $G_v=G(F_v)$ and we put $G_\infty=\prod_{v\mid\infty}G_v$.
If $\q$ is a finite place, we write $K_\q=G(\mathcal{O}_\q)$.

Further, we fix a cuspidal automorphic representation $\pi=\otimes_v \pi_v$ of $G(\A)$, which is cohomological with respect to the trivial coefficient system.
For simplicity, we assume that the field of definition of the finite part of $\pi$ is the field of rationals.
We denote the conductor of $\pi$ by $\mathfrak{f}_\pi$.

At last, we fix a prime $\p$ of $F$ lying above $p$ and we assume that the local representation $\pi_\p$ is the Steinberg representation $\St_\p(\C)$.
Hence, the prime $\p$ divides the conductor $\mathfrak{f}_\pi$ exactly once.

\begin{Rem}
One could also allow non-trivial central characters by working with the group $GL_2$ instead of $PGL_2$.
More generally, one could consider automorphic representations of inner forms of $GL_2$, which are split at $\p$.
To keep the notation as simple as possible we stick with the special case.
But all arguments carry over easily to the more general situation.
\end{Rem}

\section{Automorphic L-invariants}
\subsection{Extension classes}\label{Steinberg}
We recall a variant of Spie\ss' construction of extensions of the continuous Steinberg representation (cf.~Section 3.7 of \cite{Sp}).
A special case of these extensions were first constructed by Breuil in \cite{Br}.

Let $R$ be a locally pro-finite ring.
In our applications $R$ will be one of the following rings: $\Z,\ \Z/p^n,\ \Z_p,\ \Q$ or $\Q_p$.
The $R$-valued continuous Steinberg representation $\St_\p^{\cont}(R)$ of $G_\p$ is given by the set of continuous $R$-valued functions on $\PP(F_\p)$ modulo constant functions.
For any discrete ring $R$ the canonical map
$$\St_\p^{\cont}(\Z)\otimes R\too \St_\p^{\cont}(R)$$
is an isomorphism. In this case we simply put $\St_\p(R)=\St_\p^{\cont}(R).$
If $\Omega$ is a field with the discrete topology, then $\St_\p(\Omega)$ is an irreducible smooth representation of $G_\p$.

For a continuous group homomorphism $\ell\colon F^\ast\to R$ we define $\widetilde{\mathcal{E}}(\ell)$ as the set of pairs $(\Phi,r)\in C(GL_2(F_\p),R)\times R$ with
$$
\Phi\left(g\cdot\begin{pmatrix}
t_1 & u \\ 0 & t_2
\end{pmatrix}\right)
=\Phi(g) + r\cdot \ell(t_1)
$$
for all $t_1,t_2\in F_\p^\ast$, $u\in F$ and $g\in G_\p$.
The group $GL_2(F_{\p})$ acts on $\widetilde{\mathcal{E}}(\ell)$ via $g.(\Phi(-),r)=(\Phi(g^{-1}\cdot-),r)$.
The subspace $\widetilde{\mathcal{E}}(\ell)_0$ of tuples of the form $(\Phi,0)$ with constant $\Phi$ is invariant.
We get an induced $G_\p$-action on the quotient $\mathcal{E}(\ell)=\widetilde{\mathcal{E}}(\ell) /\widetilde{\mathcal{E}}(\ell)_0$.

\begin{Lem}
Let $\pr\colon G_\p\to \PP(F_\p),\ g\mapsto g.\infty$, be the canonical projection.
The following sequence of $R[G_\p]$-modules is exact:
\begin{align}\label{extseq}
0\too\St_\p^{\cont}(R)\xlongrightarrow{(\pr^{\ast},0)}\mathcal{E}(\ell)\xlongrightarrow{(0,\id_R)}R\too 0
\end{align}
We write $b_{\ell}$ for the associated cohomology class in $\HH^1(G_\p,\St^{\cont}_\p(R))$.
\end{Lem}
\begin{proof}
The only non-trivial step is to show that the map $\mathcal{E}(l)\xrightarrow{(0,\id_R)}R$ is surjective.
A direct calculation shows that $(\Phi_0,1)$ with
$$\Phi_0\left(\begin{pmatrix}
	a & b \\ c & d
\end{pmatrix}\right)= \begin{cases}
l(a) & \mbox{if}\ \ord(a) < \ord(c),\\ 
l(c) & \mbox{if}\ \ord(a)\geq\ord(c),
\end{cases}
$$
defines an element of $\widetilde{\mathcal{E}}(l).$ 
\end{proof}
\begin{Rem}
The extension class considered in \cite{Sp}, Lemma 3.11, equals $2\cdot b_{\ell}$.
\end{Rem}

Let $\ord_{\p}\colon F_{\p}^{\ast}\to \Z\subseteq \Q$ be the normalized valuation.
By construction $\mathcal{E}(\ord_\p)$ is a smooth representation with rational coefficients, which is easily seen to be non-split.
It is well known that the space of smooth extension of the trivial representation with the Steinberg representation is one-dimensional (see Chapter X, Proposition 4.7 in \cite{BW}).
Thus, $\mathcal{E}(\ord_\p)$ must be isomorphic to the smooth contragradient of the representation 
$$C(\PP(F_\p),\Q)=\Ind_{B_\p}^{G_\p}\Q,$$
where $B_\p\subseteq G_\p$ denotes the Borel subgroup of upper triangular matrices.
Since the extension is non-split the unique (up to scalar) spherical vector of $\mathcal{E}(\ord_\p)$ is a generator of the representation.
By Frobenius reciprocity we get a surjective map
$$\cind_{K_\p}^{G_\p}\Q\too\mathcal{E}(\ord_\p).$$
Since the space of spherical vectors is one dimensional the standard Hecke operator $T_\p$ must act on it by multiplication by a scalar $\lambda$.
(In fact, one can check explicitly that $\lambda=\mathcal{N}(\p)+1.$)

\begin{Pro}\label{smoothext}
The sequence of $\Q[G_\p]$-modules is exact:
\begin{align*}
0\too\cind_{K_\p}^{G_\p}\Q\xlongrightarrow{T_\p-\lambda}\cind_{K_\p}^{G_\p}\Q\too\mathcal{E}(\ord_\p)\too 0
\end{align*}
\end{Pro}
\begin{proof}
Since $\cind_{K_\p}^{G_\p}\Q$ is a free $\Q[T_\p]$-module by Theorem 10 of \cite{BL}, exactness on the left follows.
For checking exactness in the middle we may pass to complex coefficients.
By Theorem 3.2 of \cite{SK} the cokernel of $T_\p-\lambda$ is exactly the smooth contragradient of $\Ind_{B_\p}^{G_\p}\C$.
\end{proof}

By direct computation Spieß shows that actually an integral version of the above lemma is true (see Lemma 3.11 (c)).
But one only needs the weaker version with rational coefficients for the definition of automorphic $\LI$-invariants.
An advantage of the above argument is that it can easily be generalized to more general reductive groups.

\subsection{Cohomology of $PGL_2$}\label{Cohomology}
We introduce the cohomology groups which we will use to define automorphic $\LI$-invariants and state their basic properties.
Most of the results of this and the subsequent section are straightforward generalizations of results in Section 5 and 6 of \cite{Sp}.
In order to keep this paper as self-contained as possible we give a sketch of all arguments with the necessary modifications.

Throughout this section we fix a ring $R$.
Let $\Div(\PP(F))$ be the free abelian group on $\PP(F)$ with its natural $G(F)$-action.
We write $\Div_0(\PP(F))$ for the kernel of the map $$\Div(\PP(F)) \to \Z,\ \sum_P m_P P \mapsto \sum_P m_P.$$
For a prime $\q$ of $F$ we let
$$\St_{\q}=\St_\q^{\cont}(\Z)=C^{0}(\PP(F_\q),\Z)/\Z$$
be the integral Steinberg representation of $G_\q$ and for a finite set $S$ of primes of $F$ we put
$$\St_S=\bigotimes_{\q\in S}\St_\q.$$

The Steinberg representation with integral coefficients $\St_\q$ is equal to the integral cohomology with compact support of the Bruhat-Tits tree of $G_\q$.
Hence, we have an exact sequence of the form
\begin{align}\label{stres}
0\too\cind_{K_\q}^{G_\p}\Z\too\left(\cind_{I_\q}^{G_\p}\Z\right)/(W_\q=-1)\too\St_\q\too 0,
\end{align}
where $I_\q\subseteq K_\q$ denotes the Iwahori subgroup of matrices, which are congruent to an upper triangular matrix modulo $\q$, and $W_\q$ is the Atkin-Lehner involution at $\q$.

Given an $R$-module $N$, finite sets $S_0\subseteq S$ of primes of $F$ and a compact, open subgroup $K\subseteq G(\A\sinfty)$ we define $\Ah(K,S_0;N)^{S}$ to be the space of all functions
\begin{align*}
\Phi\colon G(\A\sinfty)\too\Hom_\Z(\St_{S_0},N)
\end{align*}
such that $\Phi(gk)=\Phi(g)$ for all $k\in K$ and $g\in G(\A\sinfty)$.
Furthermore, we define
$$\Ah_c(K,S_0;N)^{S}=\Hom_\Z(\Div_0(\PP(F)),\Ah(K,S_0;N)^{S}).$$

\begin{Def}
For a locally constant character $\epsilon\colon G_\infty\to \{\pm 1\}$ and a set of data as above we define
\begin{align*}
\MS^{i}(K,S_0;N)^{S,\epsilon} &= \HH^{i+r+s}(G(F),\Ah(K,S_0,N)^{S}(\epsilon))\\
\intertext{and}
\MS_c^{i}(K,S_0;N)^{S,\epsilon} &= \HH^{i+r+s-1}(G(F),\Ah_c(K,S_0,N)^{S}(\epsilon)).
\end{align*}
Furthermore, we consider the direct sum
$$\MS^{\ast}(K,S_0;N)^{S,\epsilon}=\bigoplus_{i\in\Z}\MS^{i}(K,S_0;N)^{S,\epsilon}.$$
\end{Def}
Let $\n\subseteq \mathcal{O}$ be a non-zero ideal.
If $K=K_0(\n)^S\subseteq G(\A\sinfty)$ is the subgroup of all integral matrices, which are congruent to a upper triangular matrix modulo $\n$, we set
$$\MS^{i}(\n,S_0;N)^{S,\epsilon}=\MS^{i}(K_0(\n)^S,S_0;N)^{S,\epsilon}$$
and similarly for the version with compact support.

The short exact sequence 
\begin{align}\label{divres}
0\too \Div_0(\PP(F)) \too \Div(\PP(F)) \too \Z\too 0
\end{align}
induces a short exact sequence
$$0\too\Ah(K,S_0;N)^{S}\too \Hom_\Z(\Div(\PP(F)),\Ah(K,S_0;N)^{S})\too \Ah_c(K,S_0;N)^{S}\too 0$$
and thus a long exact sequence in cohomology.
We write
\begin{align}\label{connect}
\delta\colon\MS^{i}_c(K,S_0;N)^{S,\epsilon}\too \MS^{i}(K,S_0;N)^{S,\epsilon}
\end{align}
for the connecting homomorphism.

\begin{Pro}\label{FlachundNoethersch}
Given a set of data as above we have:
\begin{enumerate}[(a)]
\item\label{FuN} The $R$-module $\MS_{?}^i(K,S_0;R)^{S,\epsilon}$ is finitely generated for $?\in\left\{\emptyset,c\right\}$ and all $i\in\Z$ if $R$ is Noetherian.
\item\label{FN2} If $N$ is a flat $R$-module, then the canonical map
	 $$\MS_?^i(K,S_0;R)^{S,\epsilon} \otimes_R N \too \MS_?^i(K,S_0;N)^{S,\epsilon}$$
	is an isomorphism for $?\in\left\{\emptyset,c\right\}$ and all $i\in\Z$.
\end{enumerate}
\end{Pro}
\begin{proof}
In the case $?=c$ and $F$ totally real this is Proposition 5.6 of \cite{Sp}.

For every $\q\in S_0$ the exact sequence \eqref{stres} induces an exact sequence of the form
$$0\to\Ah(K,S_0;N)^{S}\to(\Ah(K\times I_\q,S_0\setminus\left\{\q\right\};N)^{S})^{W_\q=-1}
\to \Ah(K\times K_\q,S_0\setminus\left\{\q\right\};N)^{S}\to 0.$$
By considering the long exact sequence in cohomology it suffices to prove both claims in the case $S_0=\emptyset.$

By Shapiro's Lemma and strong approximation we get isomorphisms
$$\MS^i(K,\emptyset;R)^{S,\epsilon}=\bigoplus_{i=1}^{n}H^{i+r+s}(\Gamma_i,R)$$
with $S$-arithmetic subgroups $\Gamma_i\subseteq G(F)$, which are independent of $R$.
Since $S$-arithmetic subgroups are of type (VFL), both claims follows for $?=\emptyset$ (see \cite{Se2}).
The case $?=c$ follows similarly by considering the long exact sequence associated to \eqref{divres}.
\end{proof}

Since by Proposition \ref{FlachundNoethersch} \eqref{FuN} the projective system $(\MS^i(K,S_0;\Z/p^n)^{S,\epsilon})_n$ (and also the version with compact support) fulfills the Mittag-Leffler condition we get the following
\begin{Cor}\label{limits}
The canonical map
$$\MS_?^i(K,S_0;\Z_p)^{S,\epsilon}\too \varprojlim_n \MS_?^i(K,S_0;\Z/p^n)^{S,\epsilon}$$
is an isomorphism for $?\in\left\{\emptyset,c\right\}$ and all $i\in\Z$.
\end{Cor}

\subsection{Automorphic L-invariants}\label{Definition}
In the following we will define $\LI$-invariants for our fixed automorphic representation $\pi$.
We show that the definition does not depend on whether we take cohomology with or without compact support.
At the end of the section we will state our main conjecture.

Let $\mathbb{T}$ the $\Z$-algebra generated by all good Hecke operators away from $p$, i.e. the polynomial ring on the variables $T_\q$ with $\q\nmid p\mathfrak{f}_\pi$.
In particular, we do not include Hecke operators at our fixed prime $\p$.
For $\q\nmid \mathfrak{f}_\pi$ let $\lambda_\q$ be the eigenvalue of $T_\q$ on a spherical vector of $\pi_\q$.
Given a $\T$-module $M$ we put
$$M[\pi]=\left\{ m\in M \mid T_\q(m)= \lambda_\q\cdot m\ \forall \q \right\}.$$

Evaluation at a normalized Iwahori-fixed vector yields a Hecke-equivariant map
\begin{align}\label{Iwahori}
\MS_?^i(\mathfrak{f}_\pi,\left\{\p\right\};R)^{\left\{\p\right\},\epsilon}\too \MS_?^i(\mathfrak{f}_\pi,\emptyset;R)^{\epsilon}
\end{align}
for any ring $R$.

\begin{Pro}\label{dimensions}
Let $\Omega$ be a field of characteristic $0$.
We have:
\begin{enumerate}[(a)]
\item\label{firstdim} For all $i\geq 0$ the connecting homomorphism \eqref{connect} induces an isomorphism
$$\delta\colon \MS_c^i(\mathfrak{f}_\pi,\emptyset;\Omega)^{\epsilon}[\pi]\too \MS^i(\mathfrak{f}_\pi,\emptyset;\Omega)^{\epsilon}[\pi]$$
and we have
 $\dim \MS_?^i(\mathfrak{f}_\pi,\emptyset;\Omega)^{\epsilon}[\pi]
=\binom{s}{i}$ for $?\in\left\{\emptyset,c\right\}$.
\item\label{seconddim} The map \eqref{Iwahori} induces an isomorphism
$$\MS_?^i(\mathfrak{f}_\pi,\left\{\p\right\};\Omega)^{\left\{\p\right\},\epsilon}[\pi]\too \MS_?^i(\mathfrak{f}_\pi,\emptyset;\Omega)^{\epsilon}[\pi]$$
for $?\in\left\{\emptyset,c\right\}$ and all $i\geq 0$.
\item\label{thirddim} For all $i\geq 0$ the connecting homomorphism \eqref{connect} induces an isomorphism
$$\delta\colon \MS_c^i(\mathfrak{f}_\pi,\left\{\p\right\};\Omega)^{\left\{\p\right\},\epsilon}[\pi]\too \MS^i(\mathfrak{f}_\pi,\left\{\p\right\};\Omega)^{\left\{\p\right\},\epsilon}[\pi]$$
and we have $\dim \MS_?^i(\mathfrak{f}_\pi,\left\{\p\right\};\Omega)^{\left\{\p\right\},\epsilon}[\pi]
=\binom{s}{i}$ for $?\in\left\{\emptyset,c\right\}$.
\end{enumerate}
\end{Pro}
\begin{proof}
The last claim follows directly from the first two.
The first claim follows from strong multiplicity one and the standard calculation of the contribution of cuspidal representations to cohomology (see for example \cite{Ha}).
Note that we assume that the finite part of $\pi$ can be defined over the rationals.
 
The second claim is a generalization of \cite{Sp}, Proposition 4.8:
as explained in the proof of Lemma \ref{FlachundNoethersch} we have a long exact sequence in cohomology of the form
\begin{align*}
\ldots
\to \MS_?^i(\mathfrak{f}_\pi,\left\{\p\right\};\Omega)^{\left\{\p\right\},\epsilon} 
\to (\MS_?^i(\mathfrak{f}_\pi,\emptyset;\Omega)^{\epsilon})^{W_\p=-1}
\to \MS_?^i(\p^{-1}\mathfrak{f}_\pi,\emptyset;\Omega)^{\epsilon}
\to\ldots
\end{align*}

Since $\pi_\p=\St_\p(\C)$ has no spherical vector, strong multiplicity one for $\PGL_2$ implies that
$$\MS_?^\ast(\p^{-1}\mathfrak{f}_\pi,\emptyset;\Omega)^{\epsilon}[\pi]=0.$$
Therefore, after taking $\pi$-isotypic components the first arrow in the above long exact sequence becomes an isomorphism.
By strong multiplicity one and the fact that $W_\p$ acts on the Iwahori fixed subspace of $\pi_\p$ by multiplication with $-1$ the canonical embedding
$$(\MS_?^i(\mathfrak{f}_\pi,\emptyset;\Omega)^{\epsilon})^{W_\p=-1}[\pi]\into (\MS_?^i(\mathfrak{f}_\pi,\emptyset;\Omega)^{\epsilon})[\pi]$$
is an isomorphism for all $i$.
\end{proof}

Every element in $\Hom(\St_{\p},\Z_p)$ can be uniquely extended to a continuous functional on $\St^{\cont}_{\p}(\Q_p)$.
Thus, we get a pairing
\begin{align*}
\Ah(\n,\left\{\p\right\};\Z_p)^{\left\{\p\right\}}\times \St^{\cont}_{\p}(\Q_p) \too 
\Ah(\n,\emptyset;\Z_p)^{\left\{\p\right\}}\otimes \Q_p
\end{align*}
for every non-zero ideal $\n\subseteq\mathcal{O}$.
Hence, using Proposition \ref{FlachundNoethersch} \eqref{FN2} we get a cup product pairing
\begin{align*}
\MS_?^{i}(\n,\left\{\p\right\};\Q_p)^{\left\{\p\right\},\epsilon}\times \HH^{j}(G(F),\St_{\p}^{\cont}(\Q_p))
\xlongrightarrow{\cup}\MS^{i+j}_?(\n,\emptyset;\Q_p)^{\left\{\p\right\},\epsilon}.
\end{align*}

Let $\ell\colon F_{\p}^{\ast}\to\Q_p$ be a continuous homomorphism and $b_{\ell}\in \HH^{1}(G(F),\St^{\an}_{\p}(\Q_p))$ be the (restriction of the) cohomology class associated to it in Section \ref{Steinberg}. 
Taking the cup product with $b_\ell$ induces maps
\begin{align*}
c^{(i)}_\ell(\pi)_?^{\epsilon}\colon \MS_?^{i}(\mathfrak{f}_\pi,\left\{\p\right\};\Q_p)^{\left\{\p\right\},\epsilon}[\pi]&\xrightarrow{\_\cup b_{\ell}} \MS_?^{i+1}(\mathfrak{f}_\pi,\emptyset;\Q_p)^{\left\{\p\right\},\epsilon}[\pi]\\
\intertext{and}
c^{(\ast)}_\ell(\pi)^{\epsilon}\colon \MS^{\ast}(\mathfrak{f}_\pi,\left\{\p\right\};\Q_p)^{\left\{\p\right\},\epsilon}[\pi]&\xrightarrow{\_\cup b_{\ell}} \MS^{\ast}(\mathfrak{f}_\pi,\emptyset;\Q_p)^{\left\{\p\right\},\epsilon}[\pi]
\end{align*}
In addition, we define
\begin{align*}
c^{(\ast)}_\ell(\pi)=\oplus_{\epsilon}c^{(\ast)}_\ell(\pi)^{\epsilon}.
\end{align*}

\begin{Lem}
Let $\ord_{\p}\colon F_{\p}^{\ast}\to \Q\subseteq \Q_p$ be the normalized valuation.
Then $c^{(i)}_{\ord_{\p}}(\pi)_?^{\epsilon}$ is an isomorphism for $?\in\left\{\emptyset,c\right\}$, every sign character $\epsilon$  and all $i$.
\end{Lem}
\begin{proof}
In the case $F$ totally real and $i=0$ this is \cite{Sp}, Lemma 6.2 \textit{(b)}.

Let $\mathcal{D}_{K_\p,?}$ be the $G(F)$-module defined by
$$0\too\mathcal{D}_{K_\p,?}
\too\Ah_?(\p^{-1}\mathfrak{f}_\pi,\emptyset;\Q)
\xrightarrow{T_\p-(\mathcal{N}(\p)+1)} \Ah_?(\p^{-1}\mathfrak{f}_\pi,\emptyset;\Q)
\too 0. $$
Via the same arguments as in the proof of the proposition above we see that
$$\HH^{i}(G(F),\mathcal{D}_{K_\p,?}(\epsilon))[\pi]=0$$
for all choices of $i$, $?$ and $\epsilon$.

By Proposition \ref{smoothext} we also have a long exact sequence of the form
$$0\too \Ah_?(\mathfrak{f}_\pi,\emptyset;\Q)^{\left\{\p\right\}}
\too \mathcal{D}_{K_\p,?}
\too\Ah_?(\mathfrak{f}_\p,\left\{\p\right\};\Q)^{\left\{\p\right\}}
\too 0$$
and the associated boundary map in cohomology is given by taking the cup product with $b_{\ord_\p}$.
Therefore, the assertion follows.
\end{proof}

\begin{Def}\label{defi}
The $i$-th (automorphic) $\LI$-Invariant $\LI_{\ell}^{(i)}(\pi,\p)^{\epsilon}\in\Q_p$ of $\pi$ at $\p$ with respect to $\ell$ and sign $\epsilon$ is defined by
$$\LI_{\ell}^{(i)}(\pi,\p)^{\epsilon}=\det ((c^{(i)}_{\ord_{\p}}(\pi)^{\epsilon})^{-1} \circ c^{(i)}_{\ell}(\pi)^{\epsilon} ).$$
Similarly, there is a variant using cohomology with compact support:
$$\LI_{\ell}^{(i)}(\pi,\p)_c^{\epsilon}=\det ((c^{(i)}_{\ord_{\p}}(\pi)^{\epsilon}_c)^{-1} \circ c^{(i)}_{\ell}(\pi)^{\epsilon}_c ).$$
\end{Def}

\begin{Rem}\label{onedimensional}
In the case $i=0$ the spaces of modular symbols, which are used to define automorphic $\LI$-invariants, are one-dimensional by Proposition \ref{dimensions} \eqref{thirddim}. Thus, we have the equality
$$c^{(0)}_{\ell}(\pi)_?^{\epsilon}= \LI_{\ell}^{(i)}(\pi,\p)_?^{\epsilon} \cdot c^{(i)}_{\ord_\p}(\pi)_?^{\epsilon}.$$
\end{Rem}

The following proposition shows that it does not matter whether we use cohomology with or without compact support for defining $\LI$-invariants.
\begin{Pro}\label{compact}
We have:
$$\LI_{\ell}^{(i)}(\pi,\p)^{\epsilon}_c=\LI_{\ell}^{(i)}(\pi,\p)^{\epsilon}.$$
\end{Pro}
\begin{proof}
Taking cup products and connecting homomorphisms in long exact sequences commute and therefore, the following diagram is commutative:
\begin{center}
\begin{tikzpicture}
    \path 	(0,0) 	node[name=A]{$\MS_c^{i}(\mathfrak{f}_\pi,\left\{\p\right\};\Q_p)^{\left\{\p\right\},\epsilon}[\pi]$}
		(6,0) 	node[name=B]{$\MS_c^{i+1}(\mathfrak{f}_\pi,\emptyset;\Q_p)^{\left\{\p\right\},\epsilon}[\pi]$}
		(0,-1.5) 	node[name=C]{$\MS^{i}(\mathfrak{f}_\pi,\left\{\p\right\};\Q_p)^{\left\{\p\right\},\epsilon}[\pi]$}
		(6,-1.5) 	node[name=D]{$\MS^{i+1}(\mathfrak{f}_\pi,\emptyset;\Q_p)^{\left\{\p\right\},\epsilon}[\pi]$};
    \draw[->] (A) -- (C) node[midway, left]{$\delta$};
    \draw[->] (A) -- (B) node[midway, above]{$c^{(i)}_\ell(\pi)_c^{\epsilon}$};
    \draw[->] (C) -- (D) node[midway, above]{$c^{(i)}_\ell(\pi)^{\epsilon}$};
    \draw[->] (B) -- (D) node[midway, right]{$\delta$};
\end{tikzpicture} 
\end{center}
The claim follows immediately since the first vertical arrow is an isomorphism by Proposition \ref{dimensions} \eqref{thirddim}.
\end{proof}

\begin{Rem}
Previous definitions of automorphic $\LI$-invariants (see for example \cite{D}, \cite{Sp}, \cite{Ge2}) used cohomology with compact support.
This is necessary in order to connect these $\LI$-invariants to derivatives of $p$-adic $L$-functions.
On the other hand, Venkatesh formulated his conjecture for regular cohomology.
Due to the above proposition it does not matter whether we work with cohomology with or without compact support.
\end{Rem}

Let us now formulate our main conjectures on automorphic $\LI$-invariants.

\begin{theConj}
For every continuous homomorphism $\ell\colon F_\p^{\ast}\to \Q_p$ we conjecture that
\begin{enumerate}[(I)]
\item\label{Co1} the following equality holds:
$$\LI_{\ell}^{(i)}(\pi,\p)^{\epsilon}=(\LI_{\ell}^{(0)}(\pi,\p)^{\epsilon})^{\binom{s}{i}},$$
\item\label{Co2} more precisely, the following equality holds:
$$c^{(\ast)}_{\ell}(\pi)^{\epsilon} = \LI_{\ell}^{(0)}(\pi,\p)^{\epsilon} \cdot c^{(\ast)}_{\ord_{\p}}(\pi)^{\epsilon},$$
\item\label{Co3} that the $i$-th $\LI$-invariant $\LI_{\ell}^{(i)}(\pi,\p)^{\epsilon}$ does not dependent on the character $\epsilon$
\item\label{Co4} and more generally, that the following equality holds
$$c^{(\ast)}_{\ell}(\pi)= \LI_{\ell}^{(0)}(\pi,\p)^{\epsilon} \cdot c^{(\ast)}_{\ord_{\p}}(\pi)$$
for one (and thus every) sign character $\epsilon$.
\end{enumerate}
\end{theConj}

Conjecture \eqref{Co1} follows from Conjecture \eqref{Co2} by Proposition \ref{dimensions} \eqref{thirddim}.
Assuming that Conjecture \eqref{Co1} (respectively Conjecture \eqref{Co2}) holds it is enough to check Conjecture \eqref{Co3} for the $0$-th $\LI$-invariant to get the full Conjecture \eqref{Co3} (respectively Conjecture \eqref{Co4}).

\begin{Rem}
The independence of the $0$-th automorphic $\LI$-invariants of the sign character is known in some cases.
In the case $F=\Q$ this was shown by Bertolini, Darmon and Iovita if the automorphic representation admits a certain Jacquet-Langlands transfer (see Theorem 6.8 of \cite{BDI}) and by Breuil in general (see Corollary 5.1.3 of \cite{Br2}).

Let us consider the special homomorphism $\ell=\log_p \circ \No_{F_\p/\Q_p}$.
In \cite{Ge2} it is shown that $\LI_{\ell}^{(0)}(\pi,\p)^{\epsilon}_c$, and therefore $\LI_{\ell}^{(0)}(\pi,\p)^{\epsilon}$, is independent of the sign character $\epsilon$ as long as a certain simultaneous non-vanishing hypotheses for automorphic $L$-functions holds.
If the prime $\p$ is of degree $1$, this implies the independence of the sign character for the $0$-th $\LI$-invariant for an arbitrary homomorphism $\ell$.
\end{Rem}
\subsection{An approximation lemma}\label{Approximation}
Automorphic $\LI$-invariants are constructed using cohomology with characteristic $0$ coefficients.
On the other hand, the derived Hecke algebra is constructed by using cohomology with finite coefficients.
In order to show compatibility between $\LI$-invariants and derived Hecke operators we show that the cohomology classes, which we used to define $\LI$-invariants, can be approximated by classes with $\Z/p^{n}\Z$-coefficients.

Let $\bar{\ell}\colon F_{\p}^{\ast}\to\Z/p^n$ be a locally constant homomorphism.
By taking cup products with the classes $b_{\bar{\ell}}$ we get maps
\begin{align}
c_{\bar{\ell}}^{\epsilon}\colon\MS^{\ast}(\mathfrak{f}_\pi,\left\{\p\right\};\Z/p^n)^{\left\{\p\right\},\epsilon}\xlongrightarrow{\_\cup b_{\bar{\ell}}}\MS^{\ast}(\mathfrak{f}_\pi,\emptyset;\Z/p^n)^{\left\{\p\right\},\epsilon}.
\end{align}
Now, let $\ell\colon F_{\p}^{\ast}\to\Z_p$ be a continuous homomorphism and write $\ell_n\colon F_{\p}^{\ast}\to\Z/p^n$ for its reduction modulo $p^n$.
Then the maps $c_{\ell_n}^{\epsilon}$ form a compatible system and, therefore, they induce a map on projective limits:
\begin{align}\label{projlim}
\varprojlim_n c_{\ell_n}^{\epsilon}\colon\varprojlim_n\MS^{\ast}(\mathfrak{f}_\pi,\left\{\p\right\};\Z/p^n)^{\left\{\p\right\},\epsilon}\too\varprojlim_n \MS^{\ast}(\mathfrak{f}_\pi,\emptyset;\Z/p^n)^{\left\{\p\right\},\epsilon}.
\end{align}
By Corollary \ref{limits} the canonical map
$$\MS^{\ast}(\n,S_0;\Z_p)^{\left\{\p\right\},\epsilon}\too \varprojlim_n \MS^{\ast}(\n,S_0;\Z/p^n)^{\left\{\p\right\},\epsilon}$$
for $S_0=\left\{\p\right\}$ or $S_0=\emptyset$ is an isomorphism.
Thus, by inverting $p$ and invoking Proposition \ref{FlachundNoethersch} \eqref{FN2} the maps \eqref{projlim} induce a homomorphism
$$\varprojlim_n c^{(\ast)}_{\ell_n}(\pi)^{\epsilon}\colon \MS^{(\ast)}(\mathfrak{f}_\pi,\left\{\p\right\};\Q_p)^{\left\{\p\right\},\epsilon}[\pi]\too\MS^{i+1}(\mathfrak{f}_\pi,\emptyset;\Q_p)^{\left\{\p\right\},\epsilon}[\pi].$$
The following lemma follows directly from the definitions.
\begin{Lem}\label{app}
For every continuous character $\ell\colon F_{\p}^{\ast}\to\Z_p$ we have:
$$c^{(\ast)}_{\ell}(\pi)^{\epsilon}=\varprojlim_n c^{(\ast)}_{\ell_n}(\pi)^{\epsilon}.$$
\end{Lem}

\section{Derived Hecke algebra}
\subsection{A general strategy}\label{algebra}
In this section we show that the existence of a generalized Hecke algebra fulfilling certain four properties implies Conjectures \eqref{Co1} and \eqref{Co2}.
In subsequent sections we construct a candidate for this algebra using derived Hecke algebras and show that it satisfies the first three properties.
Finally, we show how the crucial fourth property is connected to a conjecture of Venkatesh.

Let $\tilde{\mathbb{T}}=\oplus_{j\geq 0}\tilde{\mathbb{T}}_j$ be a graded $\Q_p$-algebra such that
\begin{enumerate}[(DH1)]
\item\label{propone} There is a ring homomorphism $\mathbb{T}\to \tilde{\mathbb{T}}_0$ such that its image lies in the center of $\tilde{\mathbb{T}}$
\item\label{proptwo} There are graded $\tilde{\mathbb{T}}$-actions on $\MS^{\ast}(\mathfrak{f}_\pi,\left\{\p\right\};\Q_p)^{\left\{\p\right\},\epsilon}$ and $\MS^{\ast}(\mathfrak{f}_\pi,\emptyset;\Q_p)^{\left\{\p\right\},\epsilon}$, which extend the $\mathbb{T}$-action
\end{enumerate}

\noindent Since the image of $\mathbb{T}$ is in the center of $\tilde{\mathbb{T}}$ by property (DH\ref{propone}) the $\tilde{\mathbb{T}}$-actions of property (DH\ref{proptwo}) descend to actions on $\pi$-isotypical components $\MS^{\ast}(\mathfrak{f}_\pi,\left\{\p\right\};\Q_p)^{\left\{\p\right\},\epsilon}[\pi]$ respectively $\MS^{\ast}(\mathfrak{f}_\pi,\emptyset;\Q_p)^{\left\{\p\right\},\epsilon}[\pi]$.
Thus, it makes sense to consider the following two properties:

\begin{enumerate}[(DH1)]\setcounter{enumi}{2}
\item\label{propthree} For every continuous homomorphism $\ell\colon F_\p^{\ast}\to \Q_p$, for all $t\in \tilde{\mathbb{T}}_j$ and all $i$ the diagram
\begin{center}
\begin{tikzpicture}
    \path 	(0,0) 	node[name=A]{$\MS^i(\mathfrak{f}_\pi,\left\{\p\right\};\Q_p)^{\left\{\p\right\},\epsilon}[\pi]$}
		(6,0) 	node[name=B]{$\MS^{i+j}(\mathfrak{f}_\pi,\left\{\p\right\};\Q_p)^{\left\{\p\right\},\epsilon}[\pi]$}
		(0,-1.5) 	node[name=C]{$\MS^{i+1}(\mathfrak{f}_\pi,\emptyset;\Q_p)^{\left\{\p\right\}\epsilon}[\pi]$}
		(6,-1.5) 	node[name=D]{$\MS^{i+j+1}(\mathfrak{f}_\pi,\emptyset;\Q_p)^{\left\{\p\right\},\epsilon}[\pi]$};
    \draw[->] (A) -- (C) node[midway, left]{$c^{(i)}_\ell(\pi)^{\epsilon}$};
    \draw[->] (A) -- (B) node[midway, above]{$t$};
    \draw[->] (C) -- (D) node[midway, above]{$(-1)^j\ t$};
    \draw[->] (B) -- (D) node[midway, right]{$c^{(i+j)}_\ell(\pi)^{\epsilon}$};
\end{tikzpicture} 
\end{center}
is commutative, i.e.~the map $c^{(\ast)}_\ell(\pi)^{\epsilon}$ is a graded $\tilde{\mathbb{T}}$-module homomorphism of degree $1$.
\item\label{propfour} $\MS^{0}(\mathfrak{f}_\pi,\left\{\p\right\};\Q_p)^{\left\{\p\right\},\epsilon}[\pi]$ generates $\MS^{\ast}(\mathfrak{f}_\pi,\left\{\p\right\};\Q_p)^{\left\{\p\right\},\epsilon}[\pi]$ as a $\tilde{\mathbb{T}}$-module.
\end{enumerate}

\begin{MLem}\label{ML}
Suppose there exists a graded $\Q_p$-algebra $\tilde{\mathbb{T}}$ which satisfies properties \emph{(DH\ref{propone}) - (DH\ref{propfour})}.
Then Conjecture \eqref{Co2} (and therefore also Conjecture \eqref{Co1}) holds.
\end{MLem}
\begin{proof}
Let $m$ be an element of $\MS^{i}(\mathfrak{f}_\pi,\left\{\p\right\};\Q_p)^{\left\{\p\right\},\epsilon}[\pi]$.
By property (DH\ref{propfour}) there exists $m_0\in \MS^{0}(\mathfrak{f}_\pi,\left\{\p\right\};\Q_p)^{\left\{\p\right\},\epsilon}[\pi]$ and an operator $t\in \tilde{\mathbb{T}}_i$ such that $t\cdot m_0=m.$

By Remark \ref{onedimensional} and property (DH\ref{propthree}) we get
\begin{align*}
c^{(i)}_\ell(\pi)^{\epsilon}(m)
&= c^{(i)}_\ell(\pi)^{\epsilon}(t\cdot m_0) \\
&= (-1)^{j}\ t\cdot c^{(0)}_\ell(\pi)^{\epsilon}(m_0) \\
&= (-1)^{j}\ t\cdot  \LI_{\ell}^{(0)}(\pi,\p)^{\epsilon} \cdot c^{(0)}_{\ord_\p}(\pi)^{\epsilon}(m_0) \\
&=\LI_{\ell}^{(0)}(\pi,\p)^{\epsilon}\cdot c^{(i)}_{\ord_\p}(\pi)^{\epsilon}(t\cdot m_0) \\
&=\LI_{\ell}^{(0)}(\pi,\p)^{\epsilon}\cdot c^{(i)}_{\ord_\p}(\pi)^{\epsilon}(m),
\end{align*}
which is precisely the statement of Conjecture \eqref{Co2}.
\end{proof}

The subsequent additional property will be used to connect the existence of an algebra as above to conjectures of Venkatesh.
\begin{enumerate}[(DH1)]\setcounter{enumi}{4}
\item\label{propfive} There is a graded $\tilde{\mathbb{T}}$-action on $\MS^{\ast}(\mathfrak{f}_\pi,\emptyset;\Q_p)^{\epsilon}$, which extends the $\mathbb{T}$-action, such for all $t\in \tilde{\mathbb{T}}_j$ and all $i$ the diagram
\begin{center}
\begin{tikzpicture}
    \path 	(0,0) 	node[name=A]{$\MS^i(\mathfrak{f}_\pi,\left\{\p\right\};\Q_p)^{\left\{\p\right\},\epsilon}$}
		(6,0) 	node[name=B]{$\MS^{i+j}(\mathfrak{f}_\pi,\left\{\p\right\};\Q_p)^{\left\{\p\right\},\epsilon}$}
		(0,-1.5) 	node[name=C]{$\MS^i(\mathfrak{f}_\pi,\emptyset;\Q_p)^{\epsilon}$}
		(6,-1.5) 	node[name=D]{$\MS^{i+j}(\mathfrak{f}_\pi,\emptyset;\Q_p)^{\epsilon}$};
    \draw[->] (A) -- (C) node[midway, left]{$\eqref{Iwahori}$};
    \draw[->] (A) -- (B) node[midway, above]{$t$};
    \draw[->] (C) -- (D) node[midway, above]{$t$};
    \draw[->] (B) -- (D) node[midway, right]{$\eqref{Iwahori}$};
\end{tikzpicture} 
\end{center}
is commutative, i.e.~ the evaluation map \eqref{Iwahori} is a graded $\tilde{\mathbb{T}}$-module homomorphism of degree $0$.
\end{enumerate}

The following lemma is an immediate consequence of Proposition \ref{dimensions} \eqref{seconddim}.
\begin{Lem}\label{reduction}
Suppose there exists a graded $\Q_p$-algebra $\tilde{\mathbb{T}}$ which satisfies properties (DH\ref{propone}), (DH\ref{proptwo}) and (DH\ref{propfive}).
Then property (DH\ref{propfour}) is equivalent to the following property:
\begin{enumerate}[(DH1')]\setcounter{enumi}{3}
\item\label{propfourtwo} $\MS^{0}(\mathfrak{f}_\pi,\emptyset;\Q_p)^{\epsilon}[\pi]$ generates $\MS^{\ast}(\mathfrak{f}_\pi,\emptyset;\Q_p)^{\epsilon}[\pi]$ as a $\tilde{\mathbb{T}}$-module.
\end{enumerate}
\end{Lem}
Note that property (DH\ref{propfourtwo}') is a statement purely about the cohomology of arithmetic groups. There are no $\p$-arithmetic groups involved.

\begin{Rem}\label{modified}
Actually, to ensure that the algebra we construct is as large as possible we will work with slightly different modules (see Section \ref{global}).
But in order to better illustrate the argument we neglected this detail in the above exposition.
\end{Rem}

\subsection{Local derived Hecke algebra}\label{local}
We recall local derived Hecke algebras and their action on various cohomology groups.
We follow largely the exposition of Venkatesh (cf.~Section 2 of \cite{Ve}) although we rewrite it in terms of group cohomology with values in adelic function spaces.

Let $\q$ be a prime of $F$, which does not divide $p$, and $n\geq 1$ an integer.
We work in the category $\mathfrak{C}(\Z/p^n)$ of smooth $G_\q$-representations with $\Z/p^n$-coefficients.
A $\Z/p^n[G_\q]$-representation $M$ is called smooth if the stabilizer of every $m\in M$ is open in $G_\q$.

\begin{Def}
The (spherical) derived Hecke algebra at $\q$ with $\Z/p^n$-coefficients is the graded algebra
$$\Hec_{\q,\Z/p^n}=\Hec(G_\q,K_\q)_{\Z/p^n}=\Ext^{\ast}_{\mathfrak{C}(\Z/p^n)}(\Z/p^n[G_\q/K_\q],\Z/p^n[G_\q/K_q]).$$
\end{Def}

By Frobenius reciprocity we have a canonical isomorphism
\begin{align*}
\Hom_{\mathfrak{C}(\Z/p^n)}(\Z/p^n[G_\q/K_\q],\Z/p^n[G_\q/K_q])
&\xrightarrow{\cong}\Z/p^n[G_\q/K_\q]^{K\q}\\
&\xrightarrow{\cong}\Z/p^n[K_\q\backslash G_\q/K_\q].
\end{align*}
So we see that the subalgebra of $\Hec_{\q,\Z/p^n}$ of elements of degree $0$ is the classical spherical Hecke algebra of $G_\q$ with $\Z/p^n$-coefficients.

Although the above definition of the derived Hecke algebra is the conceptional most satisfying one it is easier to define its action on various cohomology groups with a description in terms of double cosets:
for a coset $x\in G_\q/K_\q$ with representative $g_x\in G_\q$ we put $K_{\q,x}=K_\q \cap g_x K_\q g_x^{-1}$.
We fix a set of representatives $[K_\q\backslash G_\q/K_\q]$ of the $K_\q$-orbits of its left action on $G_\q/K_\q$.
There is an isomorphism of graded $\Z/p^n$-modules
\begin{align}\label{heckeiso}
\bigoplus_{x\in [K_\q\backslash G_\q/K_\q]}\hspace{-2em} H^{\ast}(K_{\q,x},\Z/p^n)\too \Hec_{\q,\Z/p^n},
\end{align}
where on the left hand side we are considering continuous group cohomology of the profinite group $K_{\q,x}$.
(For a detailed discussion of this isomorphism see Section 2.3 and 2.4 of \cite{Ve}.)
Thus, we can view an element of the derived Hecke algebra as a tuple $(x,\alpha)$ with $x\in G_\q/K_\q$ and $\alpha \in H^{\ast}(K_{\q,x},\Z/p^n)$.

Let $S_0\subseteq S$ be finite sets of finite places, which do not contain the prime $\q$, and let $\n\subseteq \mathcal{O}$ be a non-zero ideal such that $\q$ does not divide $\n$.
In the following we are going to construct the Hecke operator
\begin{align*}
h_{x,\alpha}\colon \MS^{\ast}(\n,S_0;\Z/p^n)^{S,\epsilon}\too \MS^{\ast}(\n,S_0;\Z/p^n)^{S,\epsilon}
\end{align*}
associated to a tuple $(x,\alpha)$ as above.

Let $K_0(\n)_x^S\subseteq G(\A\sinfty)$ be the compact, open subgroup given by the product
\begin{align*}
K_0(\n)_x^S=K_{\q,x}\times K_0(\n)^{S\cup \left\{\q\right\}}.
\end{align*}
We have the following chain of maps
\begin{align*}
H^{\ast}(K_{\q,x},\Z/p^n)&\too H^{\ast}(K_0(\n)_x^S,\Z/p^n)\\
&\xlongrightarrow{\cong} H^{\ast}(G(\A\sinfty),\Ah(K_0(\n)_x^S,\emptyset;\Z/p^n)^{S})\\
&\too \MS^{i}(K_0(\n)_x^S,\emptyset;\Z/p^n)^{S},
\end{align*}
where the first arrow is given by inflation, the second by Shapiro's Lemma and the third by restriction.
By abuse of notation we denote the image of $\alpha$ under the above chain of maps also by $\alpha$.

The operator $h_{x,\alpha}$ is defined via the composition of the following three homomorphisms:
Firstly, the natural inclusion
\begin{align}\label{mapone}
\Ah(K_0(\n)^S,S_0;\Z/p^n)^{S}\too \Ah(K_0(\n)_x^S,S_0;\Z/p^n)^{S}
\end{align}
yields the restriction map
\begin{align*}
\res_x\colon \MS^{\ast}(\n,S_0;\Z/p^n)^{S,\epsilon}\too \MS^{\ast}(K_0(\n)_x^S,S_0;\Z/p^n)^{S,\epsilon}.
\end{align*}
Secondly, the bilinear pairing
\begin{align}\label{maptwo}
\Ah(K_0(\n)_x^S,S_0;\Z/p^n)^{S}\times \Ah(K_0(\n)_x^S,\emptyset;\Z/p^n)^{S}\to \Ah(K_0(\n)_x^S,S_0;\Z/p^n)^{S}
\end{align}
given by pointwise multiplication induces a cup product pairing on cohomology.
Therefore, taking the cup product with $\alpha$ yields a map
\begin{align*}
\MS^{\ast}(K_0(\n)_x^S,S_0;\Z/p^n)^{S,\epsilon}\xlongrightarrow{\cup\alpha} \MS^{\ast}(K_0(\n)_x^S,S_0;\Z/p^n)^{S,\epsilon}.
\end{align*}
Finally, summing over right translations by coset representatives of $K_0(\n)_x^S\backslash K_0(\n)^S$ gives a map
\begin{align}\label{mapthree}
\Ah(K_0(\n)_x^S,S_0;\Z/p^n)^{S}\too \Ah(K_0(\n)^S,S_0;\Z/p^n)^{S},
\end{align}
which in turn induces the corestriction homomorphism
\begin{align*}
\cores_x\colon \MS^{\ast}(K_0(\n)_x^S,S_0;\Z/p^n)^{S,\epsilon}\too \MS^{\ast}(\n,S_0;\Z/p^n)^{S,\epsilon}.
\end{align*}

Note that this action agrees with the usual Hecke action on elements of degree $0$.
We do not check that this indeed defines an action of the local derived Hecke algebra, i.e.~that the action is compatible with multiplication.
This has been done for the cohomology of arithmetic groups with trivial coefficients (the case $S_0=S=\emptyset$) in \cite{Ve}.
The arguments carry over to the more general case easily.
Alternatively, we could simply work with the algebra generated by all endomorphisms $h_{x,\alpha}$.

\begin{Lem}\label{commutative} [DH\ref{propone}]
Let $\q^{\prime}$ be a prime of $F$, which does not divide $p\n$ and is not contained in $S$.
Then the operators $h_{x,\alpha}$ and $T_{\q^{\prime}}$ commute.
Note that the case $\q^{\prime}=\q$ is allowed.
\end{Lem}
\begin{proof}
In the case of cohomology of arithmetic groups, i.e.~$S_0=S=\emptyset$, this is Lemma 2.8 of \cite{Ve}.
The same proof works in our more general setup.
\end{proof}

The following lemma is the essential ingredient for proving that the global derived Hecke algebra (see Section \ref{global}) fulfills most of the sought-after properties of Section \ref{algebra}.
\begin{Lem}\label{diagrams}
Suppose that $\q$ does not divide the conductor $\mathfrak{f}_\pi$ of $\pi$ and let $(x,\alpha)\in\Hec_{\q,\Z/p^n}$ be an element of degree $i$.
\begin{enumerate}[(a)]
\item\label{diagram1}[DH\ref{propthree}] Let $\bar{\ell}\colon F_{\p}^{\ast}\to\Z/p^n$ be a locally constant character.
Then, we have
$$c_{\bar{\ell}}^{\epsilon}(h_{x,\alpha}(m))= (-1)^{i} h_{x,\alpha}(c_{\bar{\ell}}^{\epsilon}(m))$$
for all $m\in \MS^{\ast}(\mathfrak{f}_\pi,\left\{\p\right\};\Z/p^n)^{\left\{\p\right\}}$.
\item\label{diagram2}[DH\ref{propfive}] The evaluation maps \eqref{Iwahori} with $R=\Z/p^n$ commute with the action of $h_{x,\alpha}$.
\end{enumerate}
\end{Lem}
\begin{proof}
\eqref{diagram1}:
Let us consider the natural pairing
\begin{align}\label{pairing}
\Ah(K,\left\{\p\right\};\Z/p^n)^{\left\{\p\right\}}\times \St_{\p}(\Z/p^n) \too 
\Ah(K,\emptyset;\Z/p^n)^{\left\{\p\right\}}
\end{align}
for any compact, open subgroup $K\subseteq G(\A^{\left\{\p\right\},\infty})$.
The diagrams
\begin{center}
\begin{tikzpicture}
    \path 	(0,0) 	node[name=A]{$\Ah(K_0(\mathfrak{f}_\pi)^{\left\{\p\right\}},\left\{\p\right\};\Z/p^n)^{\left\{\p\right\}}\times \St_\p(\Z/p^n)$}
		(6,0) 	node[name=B]{$\Ah(K_0(\mathfrak{f}_\pi)^{\left\{\p\right\}},\emptyset;\Z/p^n)^{\left\{\p\right\}}$}
		(0,-1.5) 	node[name=C]{$\Ah(K_0(\mathfrak{f}_\pi)_x^{\left\{\p\right\}},\left\{\p\right\};\Z/p^n)^{\left\{\p\right\}}\times\St_\p(\Z/p^n)$}
		(6,-1.5) 	node[name=D]{$\Ah(K_0(\mathfrak{f}_\pi)_x^{\left\{\p\right\}},\emptyset;\Z/p^n)^{\left\{\p\right\}}$};
    \draw[->] (A) -- (C) node[midway, left]{\eqref{mapone}};
    \draw[->] (A) -- (B) node[midway, above]{\eqref{pairing}};
    \draw[->] (C) -- (D) node[midway, above]{\eqref{pairing}};
    \draw[->] (B) -- (D) node[midway, right]{\eqref{mapone}};
\end{tikzpicture} 
\end{center}
and
\begin{center}
\begin{tikzpicture}
    \path 	(0,0) 	node[name=A]{$\Ah(K_0(\mathfrak{f}_\pi)_x^{\left\{\p\right\}},\left\{\p\right\};\Z/p^n)^{\left\{\p\right\}}\times \St_\p(\Z/p^n)$}
		(6,0) 	node[name=B]{$\Ah(K_0(\mathfrak{f}_\pi)_x^{\left\{\p\right\}},\emptyset;\Z/p^n)^{\left\{\p\right\}}$}
		(0,-1.5) 	node[name=C]{$\Ah(K_0(\mathfrak{f}_\pi)^{\left\{\p\right\}},\left\{\p\right\};\Z/p^n)^{\left\{\p\right\}}\times\St_\p(\Z/p^n)$}
		(6,-1.5) 	node[name=D]{$\Ah(K_0(\mathfrak{f}_\pi)^{\left\{\p\right\}},\emptyset;\Z/p^n)^{\left\{\p\right\}}$};
    \draw[->] (A) -- (C) node[midway, left]{\eqref{mapthree}};
    \draw[->] (A) -- (B) node[midway, above]{\eqref{pairing}};
    \draw[->] (C) -- (D) node[midway, above]{\eqref{pairing}};
    \draw[->] (B) -- (D) node[midway, right]{\eqref{mapthree}};
\end{tikzpicture} 
\end{center}
are obviously commutative.
Hence, the corestriction and restriction maps in cohomology commute with taking cup product with $b_{\bar{\ell}}$.

Since the two compositions
\begin{align*}
&\Ah(K_0(\mathfrak{f}_\pi)_x^{\left\{\p\right\}},\left\{\p\right\};\Z/p^n)^{\left\{\p\right\}} \times \Ah(K_0(\mathfrak{f}_\pi)_x^{\left\{\p\right\}},\emptyset;\Z/p^n)^{\left\{\p\right\}}\times \St_\p(\Z/p^n) \\
\xlongrightarrow{\eqref{maptwo}\times \id}& \Ah(K_0(\mathfrak{f}_\pi)_x^{\left\{\p\right\}},\left\{\p\right\};\Z/p^n)^{\left\{\p\right\}}\times \St_\p(\Z/p^n) \\
\xrightarrow{\eqref{pairing}}& \Ah(K_0(\mathfrak{f}_\pi)_x^{\left\{\p\right\}},\emptyset;\Z/p^n)^{\left\{\p\right\}}
\end{align*}
and
\begin{align*}
&\Ah(K_0(\mathfrak{f}_\pi)_x^{\left\{\p\right\}},\left\{\p\right\};\Z/p^n)^{\left\{\p\right\}} \times \Ah(K_0(\mathfrak{f}_\pi)_x^{\left\{\p\right\}},\emptyset;\Z/p^n)^{\left\{\p\right\}}\times \St_\p(\Z/p^n) \\
\xlongrightarrow{\cong}& \Ah(K_0(\mathfrak{f}_\pi)_x^{\left\{\p\right\}},\left\{\p\right\};\Z/p^n)^{\left\{\p\right\}} \times \St_\p(\Z/p^n) \times \Ah(K_0(\mathfrak{f}_\pi)_x^{\left\{\p\right\}},\emptyset;\Z/p^n)^{\left\{\p\right\}}\\
\xlongrightarrow{\eqref{pairing}\times \id}&\Ah(K_0(\mathfrak{f}_\pi)_x^{\left\{\p\right\}},\emptyset;\Z/p^n)^{\left\{\p\right\}} \times \Ah(K_0(\mathfrak{f}_\pi)_x^{\left\{\p\right\}},\emptyset;\Z/p^n)^{\left\{\p\right\}}\\
\xrightarrow{\eqref{maptwo}}& \Ah(K_0(\mathfrak{f}_\pi)_x^{\left\{\p\right\}},\emptyset;\Z/p^n)^{\left\{\p\right\}}
\end{align*}
agree we see by standard arguments that taking cup product with $b_{\bar{\ell}}$ respectively $\alpha$ commute up to a sign (see for example (3.6) in Chapter V of \cite{KB}).
Hence, the claim follows.

\noindent Part \eqref{diagram2} follows by a similar argument.

\end{proof}

\subsection{Global derived Hecke algebra}\label{global}
In the following we will patch together various local Hecke algebras to construct a candidate for the algebra $\tilde{\mathbb{T}}$.
As noted in \cite{Ve} the main issue is that the natural reduction maps
$$\Hec_{\q,\Z/p^{n+1}}\too \Hec_{\q,\Z/p^n}$$
are in general far away from being surjective.
Following Venkatesh we consider as a first step the subalgebra of the endomorphism ring of a certain cohomology group generated by the image of the action of all local derived Hecke algebras.
Secondly we take the projective limit over those subalgebras.

Let us first introduce the modules on which the global derived Hecke algebra acts.
Remember that for a prime $q$, that does not divide the conductor of $\pi$, we write $\lambda_\q\in \Z$ for the eigenvalue of $T_\q$ on a spherical vector of $\pi_\q$.
Let $\m\subseteq \mathbb{T}$ be the kernel of the homomorphism
$$\mathbb{T}\too \Z/p,\ T_\q \mapstoo \lambda_\q \bmod p.$$
Firstly, we put
$$A_n= \MS^{\ast}(\mathfrak{f}_\pi,\emptyset;\Z/p^n)^{\epsilon}_{\m}.$$
Here, the subscript $\m$ denotes localization.
Secondly, we define
$$B_n^{\prime}=\MS^{\ast}(\mathfrak{f}_\pi,\left\{\p\right\};\Z/p^n)^{\left\{\p\right\},\epsilon}_\m$$
and
$$B_n=B_n^{\prime}/ \ker(\ev_n),$$
where $\ev_n$ denotes the map
$B_n\xrightarrow{\eqref{Iwahori}} A_n$
induced by evaluation at an Iwahori-fixed vector.
We put
$$C_n^{\prime}=\sum_{\bar{\ell}}c_{\bar{\ell}}^{\epsilon}(B_n^{\prime})$$
where the sum runs over all locally constant homomorphism $\bar{\ell}\colon F_\p^{\ast}\to \Z/p^n$ and $c_{\bar{\ell}}^{\epsilon}$ is the (localization at $\m$ of the) map constructed in Section \ref{Approximation}.
Finally, we define
$$C_n=C_n/\sum_{\bar{\ell}}c_{\bar{\ell}}^{\epsilon}(\ker{\ev_n}).$$
Thus, we have that $\ev_n$ induces an injective map
$$\ev_n\colon B_n\too A_n$$
and every locally constant homomorphism $\bar{\ell}\colon F_\p^{\ast}\to \Z/p^n$ induces a map
$$c_{\bar{\ell}}^{\epsilon}\colon B_n\too C_n.$$


Lemma \ref{commutative} and Lemma \ref{diagrams} imply that the actions of $\Hec_{\q,\Z/p^n}$ constructed in the previous section descends to actions on the modules $A_n$, $B_n$ and $C_n$.
Furthermore, the map $\ev_n$ is a $\Hec_{\q,\Z/p^n}$-homomorphism of degree $0$ and the maps $c_{\bar{\ell}}^{\epsilon}$ are $\Hec_{\q,\Z/p^n}$-homomorphism of degree $1$.

Let $\tilde{\mathbb{T}}^{\ast}_{\Z/p^n}\subseteq \End(A_n)$ be the graded algebra generated by the actions of $\Hec_{\q,\Z/p^n}$ for all primes $\q$ of $F$ which are co-prime to $p$ and the conductor of $\pi$.
There are natural reduction maps $\tilde{\mathbb{T}}^{\ast}_{\Z/p^n}\to \tilde{\mathbb{T}}^{\ast}_{\Z/p^m}$ for $n\geq m$.
\begin{Def}
The global derived Hecke algebra is the graded $\Q_p$-algebra
$$\tilde{\mathbb{T}}=\varprojlim_n \tilde{\mathbb{T}}^{\ast}_{\Z/p^n}\otimes \Q_p.$$
\end{Def}

We put $A_\infty=\varprojlim_n A_n \otimes \Q_\p$ and similarly for $B_\infty$ and $C_\infty$.
These are the modified modules mentioned in Remark \ref{modified}.
There are canonical isomorphisms
\begin{align*}
A_\infty[\pi]&\cong \MS^{\ast}(\mathfrak{f}_\pi,\emptyset;\Q_p)^{\epsilon}[\pi],\\
B_\infty[\pi]&\cong \MS^{\ast}(\mathfrak{f}_\pi,\left\{\p\right\};\Q_p)^{\left\{\p\right\},\epsilon}[\pi]\\
\intertext{and}
C_\infty[\pi]&\cong \MS^{\ast}(\mathfrak{f}_\pi,\emptyset;\Q_p)^{\left\{\p\right\},\epsilon}[\pi].\\
\end{align*}
By construction the global derived Hecke algebra $\tilde{\mathbb{T}}$ acts on all of these modules.
By Lemma \ref{commutative} and Lemma \ref{diagrams} combined with Lemma \ref{app} it fulfills (analogues of) properties (DH\ref{propone}) - (DH\ref{propthree}) and (DH\ref{propfive}).
Therefore, the following theorem follows by applying (variants of) our Main Lemma \ref{ML} and Lemma \ref{reduction}.

\begin{Thm}\label{dastheorem}
If $\MS^{0}(\mathfrak{f}_\pi,\emptyset;\Q_p)^{\epsilon}[\pi]$ generates $\MS^{\ast}(\mathfrak{f}_\pi,\emptyset;\Q_p)^{\epsilon}[\pi]$ as a $\tilde{\mathbb{T}}$-module, then Conjecture \eqref{Co2} (and therefore also Conjecture \eqref{Co1}) holds.
\end{Thm}

\begin{Rem}
The assumption in the theorem was formulated as a question by Venkatesh in \cite{Ve} for tempered automorphic representations, which are cohomological with respect to the trivial coefficient system, over arbitrary reductive groups.
Note that in \textit{loc.~cit.~}the prime $p$ is always co-prime to the conductor of the automorphic representation.
Under several assumptions, most notably the existence of Galois representations attached to tempered automorphic representations and a big image condition modulo $p$, Venkatesh answers the question affirmatively for split, simply connected, semisimple groups over $\Q$ (see \cite{Ve}, Theorem 7.6).
\end{Rem}

\bibliographystyle{abbrv}
\bibliography{bibfile}

\def\cprime{$'$}
\begin{thebibliography}{10}

\bibitem{BWi}
D.~{Barrera} and C.~{Williams}.
\newblock {$\mathcal{L}$-invariants and exceptional zeros of {B}ianchi modular
  forms}.
\newblock {\em Trans. Amer. Math. Soc.}, 372(1):1--34, 2019.

\bibitem{BL}
L.~Barthel and R.~Livn{\'e}.
\newblock Modular representations of {${\rm GL}_2$} of a local field: the
  ordinary, unramified case.
\newblock {\em J. Number Theory}, 55(1):1--27, 1995.

\bibitem{BDI}
M.~Bertolini, H.~Darmon, and A.~Iovita.
\newblock Families of automorphic forms on definite quaternion algebras and
  {T}eitelbaum’s conjecture.
\newblock {\em Astérisque}, 331:29--64, 2010.

\bibitem{BW}
A.~Borel and N.~Wallach.
\newblock {\em Continuous Cohomology, Discrete Subgroups, and Representations
  of Reductive Groups}.
\newblock Mathematical surveys and monographs. American Mathematical Society,
  2000.

\bibitem{Br}
C.~Breuil.
\newblock Invariant {$L$} et série spéciale p-adique.
\newblock {\em Annales scientifiques de l'École Normale Supérieure},
  37(4):559--610, 2004.

\bibitem{Br2}
C.~Breuil.
\newblock Série spéciale p-adique et cohomologie étale complétée.
\newblock {\em Astérisque}, 331:65--115, 2010.

\bibitem{KB}
K.~S. Brown.
\newblock {\em Cohomology of Groups}.
\newblock Graduate Texts in Mathematics. Springer, 1982.

\bibitem{D}
H.~Darmon.
\newblock Integration on {$H_{ p}\times H$} and arithmetic applications.
\newblock {\em Ann. of Math. (2)}, 154(3):589--639, 2001.

\bibitem{Ge2}
L.~{Gehrmann}.
\newblock {Functoriality of automorphic L-invariants and applications}.
\newblock {\em Documenta Mathematica}, to appear.

\bibitem{Ha}
G.~Harder.
\newblock Eisenstein cohomology of arithmetic groups. the case {GL}2.
\newblock {\em Inventiones mathematicae}, 89(1):37--118, 1987.

\bibitem{SK}
S.-i. Kato.
\newblock On eigenspaces of the {H}ecke algebra with respect to a good maximal
  compact subgroup of a p-adic reductive group.
\newblock {\em Mathematische Annalen}, 257(1):1--7, 1981.

\bibitem{Orton}
L.~Orton.
\newblock An elementary proof of a weak exceptional zero conjecture.
\newblock {\em Canadian Journal of Mathematics}, 56:373--405, 2004.

\bibitem{Se2}
J.-P. Serre.
\newblock Cohomologie des groupes discrets.
\newblock In {\em Séminaire Bourbaki vol. 1970/71 Exposés 382–399}, volume
  244 of {\em Lecture Notes in Mathematics}, pages 337--350. Springer Berlin
  Heidelberg, 1972.

\bibitem{Sp}
M.~Spieß.
\newblock On special zeros of p-adic {L}-functions of {H}ilbert modular forms.
\newblock {\em Inventiones mathematicae}, 196(1):69--138, 2014.

\bibitem{Ve}
A.~{Venkatesh}.
\newblock {Derived Hecke algebra and cohomology of arithmetic groups}.
\newblock {\em ArXiv e-prints}, Aug. 2016.

\end{thebibliography}

\end{document}